\documentclass[graybox]{svmult}

\usepackage{mathptmx}       
\usepackage{helvet}         
\usepackage{courier}        
\usepackage{type1cm}        
\usepackage{makeidx}         
\usepackage{graphicx}        
\usepackage{multicol}        
\usepackage[bottom]{footmisc}
\usepackage{url}
\usepackage{amssymb}

\newcommand{\C}{\mathbb{C}} 
\newcommand{\F}{\mathbb{F}} 
 
\newcommand{\PP}{\mathbb{P}} 
 
\newcommand{\R}{\mathbb{R}} 
 
\newcommand{\eqr}[1]{~\mbox{$(${\rm \ref{#1}}$)$}}
\newcommand{\rank}{\mathrm{rank}\,}
\newcommand{\vier}[4]{\left[ \begin{array}{ccc}
                   #1 &\;& #2 \\ #3 &\;& #4 \end{array} \right]}
\newcommand{\kerRs}{{\rm ker}_{\F(s)}\,}

\newcommand{\imRs}{{\rm im}_{\F(s)}\,}

\newcommand{\sep}{\;|\;}

\makeindex            
\begin{document}

\title*{Pole Placement with Fields of Positive 
   Characteristic} 
  \author{Elisa Gorla and Joachim Rosenthal} 
  \institute{Elisa Gorla \at Department of  Mathematics,
  University of Basel, Basel, Switzerland\\
  \url{http://www.math.unibas.ch} \and Joachim Rosenthal \at
  Institute of Mathematics, University
  of Z\"urich,  Z\"urich,  Switzerland\\
  \url{http://www.math.uzh.ch/aa}}

\maketitle

\abstract{The pole placement problem belongs to the classical
  problems of linear systems theory. It is often assumed that the
  ground field is the real numbers $\R$ or the complex numbers
  $\C$.\\
  The major result over the complex numbers derived in 1981 by
  Brockett and Byrnes states that arbitrary static pole placement
  is possible for a generic set of $m$-inputs, $p$-outputs and
  McMillan degree $n$ system as soon as $mp\geq n$. Moreover the
  number of solutions in the situation $mp=n$ is an
  intersection number first computed by Hermann Schubert in the
  19th century. \\
  In this paper we show that the same result with slightly
  different proofs holds over any algebraically closed field.}

\section{Introduction}               \label{sec:1}

Let $\F$ be an arbitrary field and let $A,B,C$ be matrices of size
$n\times n$, $n\times m$ and $p\times n$, with entries in $\F$. These
matrices define a discrete time dynamical system through the
equations:
\begin{equation}                     \label{system}
  \begin{array}{ccl}
  x(t+1)&=&Ax(t)+Bu(t)\\
  y(t)&=&Cx(t).
  \end{array}
\end{equation}

An $m\times p$ matrix $K$ with entries in $\F$ defines the
feedback law:
\begin{equation}                   \label{feed-law}
u(t)=Ky(t).
\end{equation}
Applying\eqr{feed-law} to the system\eqr{system}, one gets the closed
loop system:
\begin{equation}
 x(t+1)=\left( A+BKC\right)x(t).
\end{equation}

The static output pole placement problem asks for conditions on
the matrices $A,B,C$ which guarantee that the characteristic
polynomial of the closed loop system, i.e., the characteristic
polynomial of the matrix $\left( A+BKC\right)$ can be made
arbitrary.

We can explain this problem also in terms of the so-called pole
placement map. For this, identify the set of monic polynomials of
degree $n$ of the form:
$$
s^n+a_{n-1}s^{n-1}+\cdots+a_1s+a_0\in \F[s]
$$
with  the vector space $\F^n$. Then we are seeking conditions
which guarantee that the pole placement map:
\begin{equation}                       \label{pole-map}
  \chi_{(A,B,C)}:\, \F^{m\times p}\longrightarrow\F^n,\hspace{3mm}
  K\longmapsto \det(sI-A-BKC)
\end{equation}
is surjective, or at least that the image contains a non-empty
Zariski-open set.

Many facets of this problem have been studied in the literature
and the reader is referred to~\cite{by89,ro98,wa92,wa96} where
also more references to the literature can be found. 

If the base field is the complex numbers, then the major result
is due to Brockett and Byrnes~\cite{br81}:
\begin{theorem}                  \label{BrBy}
  If the base field $\F$ equals $\C$, the complex numbers, then  
  $\chi$ is surjective for generic matrices $A,B,C$ if 
  and only if $mp\geq n$. 
Moreover if $mp=n$ and $\chi$ is surjective, then the general
fiber $\chi^{-1}(\phi)$ has  cardinality
\begin{equation}                       \label{cond-1}
d(m,p)= \frac{1!2!\cdots (p-1)!(mp)!}{m!(m+1)!\cdots(m+p-1)!}.
\end{equation}
\end{theorem}

In the next section we will go over the proof of
Theorem~\ref{BrBy} in the situation when the base field $\F$ is
algebraically closed and has characteristic zero. In
Section~\ref{sec:3} we will address the difficulties which
occur in positive characteristic. The main result of the paper is
a proof that Theorem~\ref{BrBy} holds over any algebraically closed
field in the case $n=mp$.

\section{Connection to Geometry and a proof 
of Theorem~\ref{BrBy}  in characteristic zero} \label{sec:2}

Consider the transfer function $G(s):=C(sI-A)^{-1}B$ and
a left coprime factorization:
$$
G(s)=D^{-1}(s)N(s)=C(sI-A)^{-1}B.
$$
Over any field $\F$ we have the property that the $p\times (m+p)$
matrix $\left[ N(s)\ D(s)\right]$ has rank $p$ when evaluated at
an arbitrary element of the algebraic closure $\bar{\F}$ of $\F$.
In other words if $\lambda\in\bar{\F}$ then
$$
\rank\left[ N(\lambda)\ D(\lambda)\right]=p.
$$
It was the insight of Hermann and Martin~\cite{ma78} to realize
that every linear system $G(s)$ naturally defines a rational map
into the Grassmann variety Grass$(p,\F^{m+p})$:
$$
h:\ \mathbb{P}^1\longrightarrow \mathrm{Grass}(p,\F^{m+p}),\
s\longmapsto \mathrm{rowsp}\left[ N(s)\ D(s)\right].
$$
The map $h$ does not depend on the coprime factorization, and two
different linear systems $G_1(s)$ and $G_2(s)$ have different
associated rational maps. By the previous remark, the map is well
defined for every element $\lambda\in\bar{\F}$.  For this reason one
usually refers to $h$ as the Hermann-Martin map associated to the
linear system $G(s)$.

In order to arrive at an algebraic geometric formulation of the
pole placement problem, consider a left coprime factorization
$G(s)=D^{-1}(s)N(s)$ with the property that $\det (sI-A)=\det
D(s)$.  Then it is well known that the closed loop characteristic
polynomial can also be written as:
\begin{equation}\label{polypol}
  \det(sI-A-BKC)=\det\vier{I}{K}{N(s)}{D(s)}.
\end{equation}
Assume now that a desired closed loop characteristic polynomial
$\phi(s)$ factors over the algebraic closure as:
$$
\phi(s)=\prod_{i=1}^n(s-s_i),\ s_i\in\bar{\F},\ i=1,\ldots,n.
$$
The condition $\det(sI-A-BKC)=\phi(s)$ then translates into the
geometric condition:
$$
\mathrm{rowsp}\left[I\ K\right]\bigcap
\mathrm{rowsp}\left[N(s_i)\ D(s_i)\right]\neq \{ 0\}, \
i=1,\ldots,n.
$$

This formulation is closely connected to a theorem due to
Hermann Schubert:
\begin{theorem} \label{Schub} Given $n$ $p$-dimensional subspaces
  $U_i\subset \C^{m+p}$. If $n\leq mp$, then there is an
  $m$-dimensional subspace $V\subset \C^{m+p}$ such that
  \begin{equation} \label{cond-2} V \bigcap U_i\neq \{ 0\}, \
    i=1,\ldots,n.
  \end{equation}
  Moreover if $n=mp$ and the subspaces $U_i$ are in ``general
  position'', then there are exactly $d(m,p)$ (see
  Equation\eqr{cond-1}) different solutions $V\subset \C^{m+p}$
  satisfying Condition\eqr{cond-2}.
\end{theorem}

Theorem~\ref{Schub} was derived by Hermann Schubert towards the
end of the 19th century~\cite{sc1886,sc1891}. The mathematicians
at the time were not convinced with the proofs Schubert was
providing. The verification of the statements constituted
Hilbert's 15th problem, which he presented at the International
Congress of Mathematics in 1900 in Paris.  Theorem~\ref{Schub}
has been later verified rigorously and we refer to Kleiman's
survey article~\cite{kl76}.

It is not completely obvious how the geometric result of Schubert
implies Theorem~\ref{BrBy} of Brockett and Byrnes. The following
questions have to be addressed: \label{rem-1}
\begin{enumerate}
\item Given an $m$-dimensional subspace $\mathrm{rowsp}\left[K_1\
    K_2\right]\subset \C^{m+p}$, where $K_1$ is an $m\times m$
  matrix and $K_2$ is an $m\times p$ matrix. Assume
  $\mathrm{rowsp}\left[K_1\ K_2\right]\subset \C^{m+p}$ is a
  geometric solution, i.e.,
  \begin{equation} \label{cond-3}
    \det\vier{K_1}{K_2}{N(s_i)}{D(s_i)}=0, \ i=1,\ldots,n.
  \end{equation}
  Does it follow that $\left[K_1\ K_2\right]$ is row equivalent
  to $\left[I\ K\right]$ and $K$ represents a feedback law? For
  this to happen it is necessary and sufficient that $K_1$
  is invertible.
\item Assume $\mathrm{rowsp}\left[K_1\ K_2\right]\subset
  \C^{m+p}$ is a geometric solution in the sense of\eqr{cond-3}.
  Does it follow that $\det\vier{I}{K}{N(s)}{D(s)}$ is NOT the
  zero polynomial?
\item How is it possible to deal with multiple roots?
\end{enumerate}

These questions were all addressed in~\cite{br81}. A key
ingredient is the notion of non-degenerate system.
\begin{definition}
  An $m$-input, $p$-output linear system $G(s)=D^{-1}(s)N(s)$ is
  called degenerate, if there exist an $m\times m$ matrix $K_1$ and an
  $m\times p$ matrix $K_2$ such that $\left[K_1\ K_2\right]$ has full
  rank $m$ and 
  \begin{equation} \label{cond-4}
    \det\vier{K_1}{K_2}{N(s)}{D(s)}=0.
  \end{equation}
  A system $G(s)$ which is not degenerate will be called
  non-degenerate.
\end{definition}
In more geometric terms, the Hermann Martin curve associated to a
non-degenerate system does not lie in any Schubert hyper-surface.

If $\left[N(s)\ D(s)\right]$ represents a non-degenerate system
of McMillan degree $n$, then $$\det\vier{K_1}{K_2}{N(s)}{D(s)}\neq
0$$ for any $\left[K_1\ K_2\right]$ of full rank.
If in addition $\left[K_1\ K_2\right]$ is a
geometric solution, then Condition\eqr{cond-3} is satisfied and
$\det\vier{K_1}{K_2}{N(s)}{D(s)}$ is a polynomial of degree at least
$n$. All the full size minors of $\left[N(s)\ D(s)\right]$ have degree
less than $n-1$, with the exception of the determinant of $D(s)$,
which has degree $n$. So the polynomial
$\det\vier{K_1}{K_2}{N(s)}{D(s)}$ cannot have degree $n$ unless $K_1$
is invertible. Hence it follows that a geometric solution for a
non-degenerate system results in a feedback solution $u=Ky$ on the
systems theory side. 

Non-degenerate systems are therefore very desirable. The
following theorem was formulated in~\cite{br81} in the case when
the base field is the complex numbers.

\begin{theorem} \label{nondegenerate} Let $\F$ be an arbitrary
  field. If $n<mp$ then every system $(A,B,C)$ defined over $\F$ with
  $m$-inputs, $p$-outputs and McMillan degree $n$ is degenerate. If
  $\F$ is an algebraically closed field and $n\geq mp$, then a generic
  system $(A,B,C)\in\F^{n^2+n(m+p)}$ is non-degenerate.
\end{theorem}

The proof of the first part of the statement follows from basic
properties of 
coprime factorizations of transfer functions. Indeed let the
$p\times (m+p)$ polynomial matrix $M(s)=\left[N(s)\ D(s)\right]$
represent an $m$-inputs, $p$-outputs system of McMillan degree
$n<mp$. Then, possibly after some row reductions, we find a
row of $M(s)$ whose degree is at most $m-1$. Using this row one
readily constructs a full rank $m\times (m+p)$ matrix such
that\eqr{cond-4} holds. This shows that $G(s)=D^{-1}(s)N(s)$ is
degenerate.

The second part of the statement, namely that a generic system defined
over $\F$ is non-degenerate, will be established through a series of
lemmas. Here $\F$ is an algebraically closed field of characteristic
zero.

Notice that it is enough to show that the set of degenerate systems
is contained in a proper algebraic set of $\F^{n^2+n(m+p)}$. In order
to prove this, we establish an algebraic relation between
the polynomial matrix $\left[N(s)\ D(s)\right]$ and the matrices
$(A,B,C)$. The following lemma is an ingredient of classical
realization theory. The proof and the concept of basis matrix is
found in~\cite{ro97a1}.

\begin{lemma} \label{lemma-1} Assume $G(s)=D^{-1}(s)N(s)$ is a
  left coprime factorization of a $p\times m$ transfer function
  of McMillan degree $n$. Then for every $p\times n$ basis matrix
  $X(s)$ there are matrices $A\in\F^{n\times n}, B\in\F^{n\times
    m}$ and $C\in\F^{p\times n}$ such that:
  \begin{equation} \label{cond-5} \kerRs [X(s) \sep N(s)\sep
    D(s)] \;=\; \imRs \left[ \begin{array}{cc} (sI_n-A) & B \\ 0
        & I_m\\ C & 0
      \end{array} \right].
  \end{equation}
  Furthermore $(A,B,C)$ is a minimal realization of $G(s)$,
  i.e.,
$$
G(s)=C(sI-A)^{-1}B,
$$ 
and for every minimal realization $(A,B,C)$ of $G(s)$ there
exists a basis matrix $X(s)$ such that\eqr{cond-5} is satisfied.
\end{lemma}

As pointed out in~\cite{ro97a1}, for certain basis
matrices $X(s)$ it is possible to compute $(A,B,C)$ just ``by
inspection''.

Using the previous lemma, one readily establishes the following:

\begin{lemma} \label{lemma-2} Assume that $(A,B,C)$ is a minimal
  realization of $G(s)=D^{-1}(s)N(s)$ and $\det(sI-A)=\det D(s)$.
  Then
  \begin{equation}
    \det\vier{K_1}{K_2}{N(s)}{D(s)}=\det\vier{sI-A}{B}{K_2 C}{K_1}.
  \end{equation}
\end{lemma}

As before, identify an $m$-inputs, $p$-outputs system $(A,B,C)$ of
McMillan degree $n$ with a point of $\F^{n^2+n(m+p)}$. Let $S$ be the set:
\begin{equation}                   \label{coincidence}
\left\{ ((K_1,K_2),(A,B,C))\in
  \mathrm{Grass}(m,\F^{m+p})\times \F^{n^2+n(m+p)} :
  \det\vier{sI-A}{B}{K_2 C}{K_1}=0\right\}.
\end{equation}
Since Grass$(m,\F^{m+p})$ is a projective
variety, the projection of $S$ onto $\F^{n^2+n(m+p)}$ is an algebraic
set. This follows from the main theorem of elimination theory (see,
e.g.,~\cite{mu76}). We have therefore established that the set of
degenerate systems inside $\F^{n^2+n(m+p)}$ is an algebraic
set. We establish the genericity result as soon as we can
show the existence of one non-degenerate system, under the assumption
that $n\geq mp$.

\begin{remark}
  In the case of proper transfer functions, the dimension of the
  coincidence set $S$ was computed in~\cite[Theorem~5.5]{ro94}.
  With this result it was then shown in~\cite{ro94} that the set
  of non-degenerate systems inside the quasi-projective variety
  of proper transfer functions contains a dense Zariski-open  
  set as soon as $n\geq mp$. 
\end{remark}

\begin{definition}
  Let $\F$ be an algebraically closed field of characteristic
  $0$. The {\em osculating normal curve} $C_{p,m}$ is the closure
  of the image of the morphism
  \begin{equation}\label{ONC}
    \begin{array}{ccl}
      \F & \longrightarrow & \mathrm{Grass}(p,\F^{m+p}) \\
      s & \longmapsto &
      \mathrm{rowsp}\left[\frac{d}{d^i}s^j\right]_{i=0,\ldots,p-1;\,
        j=0,\ldots,m+p-1.}
    \end{array}
  \end{equation}
\end{definition}
We denote by $d/d^i$ the $i$-th derivative with respect to $s$,
i.e.,
$$
\frac{d}{d^i}s^j= \left\{\begin{array}{ll}
    \prod_{k=0}^{i-1}(j-k) s^{j-i} & \mbox{if $j\geq i$} \\
    0 & \mbox{if $j<i$.}
  \end{array}\right.
$$

The osculating normal curve is an example of a non-degenerate
curve in the Grassmannian $\mathrm{Grass}(p,\F^{m+p})$.  An
elementary matrix proof of this fact was first given
in~\cite{ro90}. We will say more about it in the next section.
If $n>mp$ one constructs a non-degenerate system by simply
multiplying the last column of the matrix representing the
osculating normal curve by $s^{n-mp}$.

In the case $p=1$, this is the rational normal curve of degree
$m$ in $\PP^m\cong \mathrm{Grass}(1,\F^{m+1})$. In the case
$m=1$, the osculating normal curve is isomorphic to the rational
normal curve of degree $p$ in $\PP^p\cong
\mathrm{Grass}(p,\F^{p+1})$.

So far we have shown that if $mp\geq n$, then a generic system
is non-degenerate. Moreover, if $n=mp$, the system is non-degenerate and the
desired closed loop polynomial has distinct roots, then pole
placement is possible with $d(m,p)$ different feedback
compensators. 

It remains to be addressed the question of multiple roots in the
closed loop polynomial. This has been done in the literature by
lifting the pole placement map\eqr{pole-map} from $\F^{m\times
  p}$ to the Grassmann variety Grass$(m,\F^{m+p})$. We follow
the arguments in~\cite{ro98}.

We can expand the closed loop characteristic polynomial as:
\begin{equation}
  \det\vier{K_1}{K_2}{N(s)}{D(s)}=\sum_\alpha k_\alpha g_\alpha(s),
\end{equation}
where $k_\alpha$ are the Pl\"ucker coordinates of
$\mathrm{rowsp}\left[ K_1\
  K_2\right]\in\mathrm{Grass}(m,\F^{m+p})$ and where the
polynomials $g_\alpha(s)$ are (up to sign) the corresponding
Pl\"ucker coordinates of $[N(s)\ D(s)]$.  Let $\PP^N$ be the
projective space $\PP(\wedge^m\F^{m+p})$ and let
$$
E_{(A,B,C)}:=\left\{ k\in \PP^N \mid \sum_\alpha
  g_\alpha(s)k_\alpha =0 \right\} .
$$
As shown in~\cite{wa96}, one has an extended pole placement map
with the structure of a central projection:
\begin{equation}
  \label{central}
  L_{(A,B,C)} \; :\;
  \PP^N-E_{(A,B,C)} \  \longrightarrow \
  \PP^n,\hspace{9mm}
  k\  \longmapsto \ \sum_\alpha  k_\alpha g_\alpha(s).
\end{equation}
A system $[N(s)\ D(s)]$ is non-degenerate if and only if:
$$
E_{(A,B,C)}\cap \mathrm{Grass}(m,\F^{m+p})=\left\{ \right\}.
$$
For a non-degenerate system, the extended pole placement map 
$L_{(A,B,C)}$ induces a finite morphism:
\begin{equation} \label{pole-map-ext} \hat{\chi}_{(A,B,C)}:\,
  \mathrm{Grass}(m,\F^{m+p})\longrightarrow\PP^n,\hspace{3mm}
  \mathrm{rowsp}\left[K_1\ K_2\right] \longmapsto
  \det\vier{K_1}{K_2}{N(s)}{D(s)}.
\end{equation}

The inverse image of a closed loop polynomial $\phi(s)\in \PP^n$
under the map $L_{(A,B,C)}$ is a linear space which intersects the
Grassmann variety $\mathrm{Grass}(m,\F^{m+p})$ in as many points
(counted with multiplicity) as the degree of the Grassmann variety. 
This is equal to Schubert's number $d(m,p)$.

This completes the proof of Theorem~\ref{BrBy} of Brockett and
Byrnes in the case $n=mp$ not only for the field of complex numbers,
but also in the 
case when the base field is algebraically closed and has
characteristic zero. In Remark~\ref{n<mp} we will discuss how to extend
the proof to the case when $n<mp$.

In the next section we will discuss how to extend
Theorem~\ref{nondegenerate} to the case of an algebraically closed
field of positive characteristic.  
We will show that it is much more tricky to establish the existence of
non-degenerate systems in the case when the base field has positive
characteristic.

\section{A proof of Theorem~\ref{BrBy} in positive
  characteristic} \label{sec:3}

Let $\F$ be an algebraically closed field of characteristic
$q>0$. Lemma~\ref{lemma-1} and Lemma~\ref{lemma-2} as formulated
in the last section only depend on techniques from linear algebra
and are true over an arbitrary field, so in particular over an
algebraically closed field.

If $X$ is a projective variety, $Y$ is a quasi-projective
variety, and $S\subset X\times Y$ is an algebraic subset, then the
projection of $S$ onto $Y$ is a Zariski-closed subset of $Y$
(see, e.g.,~\cite[Chapter I, Section 5.2]{sh94b1}). This shows that
Theorem~\ref{nondegenerate} also holds over an algebraically
closed field. In order to establish Theorem~\ref{BrBy}, we have to
show that there exists at least one non-degenerate system for any
choice of the parameters $p,m,n\geq mp$.  
We will also show that a generic fiber contains $d(m,p)$ elements
when $n=mp$.

The last statement is true as soon as the extended pole placement
map $\hat{\chi}_{(A,B,C)}$ is separable~\cite[Chapter II, Section
6.3]{sh94b1}. This is indeed the case: $\hat{\chi}_{(A,B,C)}$ can
be seen as the composition of the Pl\"ucker embedding (which just
involves the computation of minors) and the linear map
$L_{(A,B,C)}$. Both maps are separable, and we conclude therefore
that the composition map is separable.

So there remains the problem of establishing the existence of
non-degenerate systems in the case $n\geq mp$. As we will show next, the
osculating normal curve may be degenerate in
characteristic $q>0$. In Section~\ref{MDS} we will provide alternative
examples of non-degenerate systems in positive characteristic, 
while in Section~\ref{ff} we will discuss the case of finite fields.

\subsection{The osculating normal curve}

Although the osculating normal curve is defined over a field of 
characteristic zero, its reduction modulo $q$ defines a curve in
$\mathrm{Grass}(p,\F^{m+p})$, which can again be regarded as the
closure of the image of morphism (\ref{ONC}).  If $p=1$, the
curve is the rational normal curve of degree $m$ in
$\mathrm{Grass}(1,\F^{m+1})\cong\PP^m$. In particular it is
non-degenerate.  Notice however that the reduction of the
osculating normal curve is degenerate whenever $q\leq p+m$, provided
that $p\geq 2$.  
This is easily checked if $q<p$, since in this case the $(q+1)$-st
row of the matrix defining the curve is identically zero. If $p\leq
q\leq p+m$, consider the minor of the sub-matrix consisting of
columns $1,\ldots,p-1,q$. This sub-matrix has the form:
$$
\left[
  \begin{array}{ccccc}
    1  &\ s & \ldots\ &  s^{p-2}    & s^q\\
    0  & 1\ &         & (p-2)s^{p-3}& 0\\
    \vdots &    &\ \ddots &   \vdots    & \vdots \\
    \vdots &    &         &  (p-2)!     & \vdots \\
    0  &\ldots&       &     0       &  0
  \end{array}
\right]
$$
It follows that the corresponding minor is zero. By choosing a
compensator $[K_1\ K_2]$ whose sub-matrix consisting of the
``complementary columns'' $p,p+1,\ldots,q-1,q+1,\ldots,p+m$ is the
identity matrix and where all other elements are zero, one
verifies that the osculating normal curve is also degenerate in
this situation.

\begin{remark} \label{zerominor} If at least one minor of the
  matrix $\left[N(s)\ D(s)\right]$ is 0, then the system
  $G(s)=D^{-1}(s)N(s)$ is degenerate.
\end{remark}

Notice that if $q\gg 0$, then the reduction modulo $q$ of the
osculating normal curve is  
non-degenerate. This reflects the usual fact that ``fields with 
large enough characteristic behave like fields of characteristic
zero''.

The appearance of many zero entries in the matrix over a field
$\F$ of ``small'' positive characteristic $q$ is due to the fact that
many derivatives vanish. More precisely, let $h\in\{0,\ldots,q-1\}$
s.t. $j=h$ mod. $q$. Then
$$
\frac{d}{d^i}s^j= \left\{\begin{array}{ll}
    \prod_{k=0}^{i-1}(j-k) s^{j-i} & \mbox{if $h\geq i$} \\
    0 & \mbox{if $h<i$}
  \end{array}\right.
$$
This was one of the reasons that motivated Hasse to introduce the
following concept.

\begin{definition}
  The {\em $i$-th Hasse derivative} of a polynomial
  $u(s)=\sum_{j=0}^d u_js^j$ is defined as:
$$
\frac{\partial}{\partial^i}u(s)=\sum_{j=i}^d {j\choose i}u_i
s^{j-i}.
$$
\end{definition}

Observe that in characteristic 0 one has
$$
\frac{\partial}{\partial^i}=\frac{1}{i!}\frac{d}{d^i}.
$$
Moreover, none of the Hasse derivatives vanishes identically for
all polynomials, regardless of the characteristic of the base
field, whereas in characteristic $q>0$, the $i$-th derivative of
any polynomial is identically zero for all $i\geq q$.

It is therefore natural that we define the osculating normal
curve in positive characteristic using the Hasse derivative
instead of the normal derivative.

\begin{definition}
  Let $\F$ be an algebraically closed field of characteristic
  $q>0$. The {\em osculating normal curve} $C_{p,m}$ is the
  closure of the image of the morphism
  \begin{equation}\label{ONCP}
    \begin{array}{ccl}
      \F & \longrightarrow & \mathrm{Grass}(p,\F^{m+p}) \\
      s & \longmapsto &
      \mathrm{rowsp}\left[\frac{\partial}{\partial^i}s^j\right]_{i=0,
        \ldots,p-1;\, j=0,\ldots,m+p-1.} 
    \end{array}
  \end{equation}
  where $\partial$ denotes the Hasse derivative.
\end{definition}

For $p\leq 2$ the definition agrees with the one given at the
beginning of this section. In particular, for $p=1$ we have a
non-degenerate rational normal curve of degree $m$ in
$\mathrm{Grass}(1,\F^{m+1})\cong\PP^m$. Notice also that the curve is
well defined even if $p>q$, as we do not generate a zero row
in the defining matrix.

Unfortunately, even with this adapted definition the osculating
normal curve is degenerate for many choices of the parameters, 
as the following result points out:

\begin{proposition}\label{degen}
  Let $\F$ be an algebraically closed field of characteristic
  $q>0$. Assume that $q\leq m$. Then the osculating normal 
  curve $C_{p,m}$ is degenerate.
\end{proposition}

\begin{proof} \smartqed
  By Remark~\ref{zerominor}, it suffices to show that one of the
  minors of the matrix:
  $$\left[{j\choose i}s^{j-i}\right]_{i=0,\ldots,p-1;\,
    j=0,\ldots,m+p-1}
  $$ 
  is zero. Consider the sub-matrix consisting of columns
  $0,\ldots,p-2,c$, where $c$  is a multiple of
  $q$, $c\in\{p+1,\ldots,p+m\}$. The corresponding minor is:
$$
\mathrm{det}\left[{j\choose i}\right]_{i=0,\ldots,p-1;\,
  j=0,\ldots,p-2,c}s={c\choose p-1}s=0.
$$ 
The first equality follows from the observation that the matrix
is upper triangular with ones on the diagonal, except for the
entry in the lower right corner which equals ${c\choose p-1}$.
\qed
\end{proof}

\begin{remark} 
  If $q\mid p$ and $m\geq p$, the minor of the sub-matrix
  consisting of columns $p-1,p+1,\ldots,2p-1$ equals
$$
\mathrm{det}\left[{j\choose i}\right]_{i=0,\ldots,p-1;\,
  j=p-1,p+1,\ldots,2p-1}s^{p^2-1}=p s^{p^2-1}=0.
$$ 
The first equality follows from Lemma~9 in~\cite{ge85}.
\end{remark}

\begin{remark}
Degeneracy of the osculating normal curve over the field $\F_q$ with
$q\leq\max\{p,m\}$ also follows from Theorem~\ref{fieldsize}.
\end{remark}

In Proposition~\ref{degen} we saw that the osculating normal
curve may be degenerate over a field $\F$ of positive
characteristic $q$. Notice however that the curve may be
non-degenerate for certain choices of the parameters $p,m$. The
following example shows, e.g., that if the field $\F$ has
characteristic $2$, $m=1$ and $p$ is odd, then $C_{p,1}$ is
non-degenerate.

\begin{example}
  The curve $C_{p,1}$ is the closure of the image of the morphism
$$
\begin{array}{ccl}
  \F & \longrightarrow & \mathrm{Grass}(p,\F^{p+1}) \\
  s & \longmapsto &
  \mathrm{rowsp}\left[{j\choose i}s^{j-i}\right]_{i=0,\ldots,p-1;\,
    j=0,\ldots,p.} 
\end{array}
$$
The minors of the matrix that defines the morphism are $${p\choose
  i}s^i\;\;\; \mbox{for $i=0,\ldots,p$.}$$ 
Hence the curve is non-degenerate if and
only if all the minors are non-zero, if and only if 
$$q\nmid {p\choose i}\;\;\; \mbox{for any $i=0,\ldots,p.$}$$
\end{example}

Over a field of even characteristic, this is in fact the only
case when the osculating normal curve is non-degenerate.

\begin{corollary}
  Let $\F$ be an algebraically closed field of characteristic
  $2$. Then the osculating normal curve $C_{p,m}$ is degenerate,
  unless $m=1$ and $p$ is odd. In the latter case, $C_{p,1}$ is
  isomorphic to the rational normal curve of degree $p$ in
  $\PP^p$.
\end{corollary}

\subsection{Monomial systems and MDS matrices}\label{MDS}

\begin{definition}\label{monsyst}
A matrix $M(s)=\left[N(s)\ D(s)\right]$ is {\em monomial} if the 
minors of all sizes of $M(s)$ are monomials.
A system $G(s)$ associated to a monomial matrix $M(s)$ is called 
a {\em monomial system}.
\end{definition}

A monomial matrix $M(s)=[\alpha_{i,j}s^{d_{i,j}}]$ is determined by:
\begin{itemize}
\item the {\em coefficient matrix} $M=[\alpha_{i,j}]$, 
\item the {\em degree matrix} $[d_{i,j}]$. 
\end{itemize}
The degree matrix has the property that
$d_{i,j}+d_{k,l}=d_{i,l}+d_{k,j}$ for all $i,j,k,l$.

\begin{example}
The osculating normal curve defines a monomial system.
\end{example}

\begin{example}
  Let $\F$ be a field which contains at least three distinct
  elements $0,1,\alpha$. The matrix
$$
M(s)=\left[\begin{array}{cccc}
    1 & 0 & s^2 & \alpha s^3 \\
    0 & 1 & s   & s^2
  \end{array}\right]
$$ 
has minors 
$$1, s, s^2, -s^2, \alpha s^3,
(1-\alpha)s^4.
$$ 
It therefore follows that $M(s)$ is a monomial matrix.  A direct
calculation shows that this system is non-degenerate.
\end{example}

\begin{definition}
A matrix $M$ with entries in $\F$ is {\em Maximum Distance Separable
  (MDS)} if all its maximal minors are non-zero.
\end{definition}

\begin{remark}
In coding theory a linear code $C\subset\F^n$ is called an MDS
code if all the maximal minors of the generator matrix of $C$ are
non-zero. This explains the choice of the name for these matrices.
\end{remark}

\begin{remark}
Let $M(s)$ be a monomial matrix. If the system associated to $M(s)$ is
non-degenerate, then $M$ is an MDS matrix. This follows from
Remark~\ref{zerominor}. It is not always the case that a monomial
matrix $M(s)$ with MDS coefficient matrix $M$ is non-degenerate.
\end{remark}

An example of degenerate $M(s)$ with MDS coefficient matrix is given
in the following example.

\begin{example}\label{MDSdeg}
Let $\F=\F_5$, the finite field of 5 elements.
 The following monomial system defined by the matrix
$$
M(s)=\left[\begin{array}{cccc}
    1 & s & s &  s^2 \\
    0 & 1 & 2 & 3s
  \end{array}\right]
$$
is left prime and 
has an MDS matrix as coefficient matrix. Nonetheless the system is
degenerate as, e.g.,
$$
\left[ K_1\ K_2\right]:= 
\left[\begin{array}{cccc}
    0 & 1 & 2 &  0 \\
    0 & 0 & 0 & 1
  \end{array}\right]
$$
results in the zero characteristic polynomial.
\end{example}

In the next theorem we show that an MDS matrix of given size
defined over a field $\F$ exists only if the ground field $\F$ has
enough elements.

\begin{theorem}\label{fieldsize}
Let $p,m\geq 2$ and let $M(s)$ be a monomial matrix of size
$p\times(m+p)$ defined over a field $\F$ with $q$ elements. If
$q\leq\max\{p,m\}$, then $M(s)$ is degenerate.
\end{theorem}

\begin{proof}
If $M(s)$ is non-degenerate, then its coefficient matrix $M$ is MDS.
Let $M^{\perp}$ be an $m\times (p+m)$-matrix defined over $\F$, such
that $\mathrm{rowsp}(M)=\mathrm{ker}(M^{\perp})$. Let $C^{\perp}$ be
the dual code of $C$. The generator matrix of $C^{\perp}$ is then
$M^{\perp}$. It is well known that $C$ is MDS if and only if
$C^{\perp}$ is MDS. Therefore, $M^{\perp}$ is an MDS matrix.

We want to show that if $M$ (resp. $M^{\perp}$) is MDS of size
$p\times (m+p)$ (resp. $m\times(m+p)$) defined over $\F_q$, then
$q\geq\max\{p+1,m+1\}$. The statement is symmetric in $p,m$, hence we
can assume without loss of generality that $2\leq p\leq m$. It
suffices to prove that $q\geq m+1$.

We first consider the case $p=2$. Since $M$ is MDS, every pair of
columns must be linearly independent. Over a field of $q$ elements,
there 
are $q^2-1$ choices for the first column, $q^2-q$ for the second,
$q^2-2(q-1)-1$ choices for the third, and so forth. Since there are
$q^2-(m+1)(q-1)-1$ choices for the $m+2$-nd column, it must be
$q^2-(m+1)(q-1)-1=(q-1)(q-m)\geq 1$, hence $q\geq m+1$.

For an arbitrary $p$, we can assume that the matrix $M$ is of the form 
$[I_p\ A]$, where $I_p$ is the $p\times p$ identity matrix and $A$ is
a matrix of size $p\times m$. The MDS property of $M$ translates into
the property that all the minors of all sizes of $A$ are
non-zero. Consider the submatrix $N$ obtained from $M$ by deleting the
last $p-2$ rows and the columns $3,\ldots,p$, $\ N=[I_2\ B]$ where $B$
consists of the first two rows of $A$. $N$ is a $2\times (m+2)$ MDS
matrix, since all the minors of all sizes of $B$ are non-zero. It
follows that $q\geq m+1$ for every $p\geq 1$.
\qed
\end{proof}

From the proposition it follows, e.g., that every monomial $M(s)$ of
size $2\times(m+2)$ defined over $\F_2$ is degenerate, unless
$m=1$. Clearly, there may be non-degenerate matrices which are not
monomial. E.g., the following is an example of a $2\times 4$
system
defined over $\F_2$ which is non-degenerate:

\begin{example}\label{2x4}
 Consider the system defined over $\F_2$  by the matrix
$$
M(s)=\left[\begin{array}{cccc}
    0 & s     & s+1 & s^2 \\
    1 & s^2+1 & 1   & s
  \end{array}\right].$$
The minors, listed in lexicographic order, are
$$
s, s+1, s^2, (s^3+s+1), s^4, s.
$$
A direct computation shows that the system is non-degenerate.
\end{example}

%
%

We conclude the paper with the main result.

\begin{theorem}
  Let $M(s)=[N(s)\ D(s)]$ be a monomial system having an MDS
  coefficient matrix $M$ of the form $M=[I_p\ R]$.  Let the
  degrees of the coefficient matrix be $d_{i,j}=j-i$ if $j\geq i$
  and zero else. Then $M(s)$ is non-degenerate of degree $mp$.
\end{theorem}

\begin{proof} \smartqed
Denote by $\alpha$ a multi-index
$\alpha=(\alpha_1,\ldots,\alpha_p)$ with the property that
$$
1\leq \alpha_1<\cdots \cdots < \alpha_p\leq m+p.
$$

Denote by $m_1(s),\ldots ,m_p(s)$ the $p$ row vectors of $M(s)$
and denote by $e_1,\ldots,e_{m+p}$ the canonical basis of
$\F^{m+p}$. One readily verifies that the Pl\"ucker expansion 
of $M(s)$ has the form:
$$
m_1(s)\wedge \ldots \wedge m_p(s)=
\sum_{\alpha\in\left\{ n\atop p\right\}}
m_\alpha e_{\alpha_1}\wedge \ldots \wedge  e_{\alpha_p} s^{|\alpha|}
$$
where $|\alpha|:=\sum_{i=1}^p(\alpha_i-i)$ and $m_\alpha$ is the
minor of $M$ corresponding to the columns
$\alpha_1,\ldots,\alpha_p$. 

The multi-indices $\alpha$ have a natural partial order, coming from
componentwise comparison of their entries. 
If $\beta =(\beta_1,\ldots,\beta_p)$ is a multi-index, then
one defines:
$$
\alpha\leq \beta \ :\Longleftrightarrow \ 
\alpha_i\leq \beta_i\mbox{ for } i=1,\ldots,p.
$$

By contradiction assume now that $M(s)$ is degenerate. Let
$[K_1\ K_2]$ be a compensator which leads to the closed loop
characteristic polynomial zero:
\begin{equation}                    \label{zero-expan}
\det\vier{K_1}{K_2}{N(s)}{D(s)}=\sum_\alpha k_\alpha g_\alpha(s)=0.
\end{equation}
In the last expansion $k_\alpha$ denotes up to sign the $m\times
m$ minor of $[K_1\ K_2]$ corresponding to the columns $1\leq
\hat{\alpha}_1<\ldots<\hat{\alpha}_m\leq (m+p)$,
$\hat{\alpha}_i\not\in\{\alpha_1,\ldots,\alpha_p\}$.

$[K_1\ K_2]$ has a well defined row reduced echelon form with
Pivot indices
$\hat{\beta}=(\hat{\beta_1},\ldots,\hat{\beta_m})$. It follows
that $k_\alpha=0$ for $\alpha\not\leq \beta$. But this means that
the term $m_\beta s^{|\beta|}$ cannot cancel in the
expansion\eqr{zero-expan} and this is a contradiction. $M(s)$ is
therefore non-degenerate. 
\qed
\end{proof}

\begin{remark}
If $n>mp$ choose $d_{i,m+p}=n-mp+m+p-i$ in order to obtain once
more a non-degenerate system of degree $n$.
\end{remark}

By establishing the existence of a non-degenerate
system, we have shown that Theorem~\ref{BrBy} holds true for
any algebraically closed field for $n=mp$.

\begin{remark}\label{n<mp}
  In order to prove Theorem~\ref{BrBy} in the situation when
  $n<mp$, one can show that for a generic system $(A,B,C)$
  the set of dependent compensators, i.e., the set of compensators
  which results in a zero closed loop characteristic polynomial,
  has minimum possible dimension, namely $mp-n-1$. This is clearly sufficient to
  establish the result. In order to prove this statement, one
  can proceed in two ways. Either one shows that the condition is
  algebraic and constructs an example of a system
  of degree $n$ satisfying the condition. Alternativeley one shows
  that the coincidence set $S$ introduced in\eqr{coincidence} has
  dimension $n^2+n(m+p)+mp-n-1$. The generic fiber of the
  projection onto the second factor has then dimension $mp-n-1$.
  This last argument was developed for the dynamic pole placement
  problem in~\cite{ro94}. 
\end{remark}

\subsection{Non-degenerate systems over finite fields}\label{ff}

In this last subsection, we show that in general non-degeneracy
does not guarantee that the pole placement map is surjective over
a finite field.

\begin{theorem}
  Let $\F_2$ be the binary field. Then 
  no non-degenerate system defined over $\F_2$ 
  induces an onto pole placement map:
  $$
  \mathrm{Grass}(2,\F_2^4)\longrightarrow\PP^4(\F_2).
  $$
\end{theorem}

\begin{proof}
  Let $M(s)$ be a non-degenerate matrix with entries in
  $\F_2[s]$. Let $\F$ denote the algebraic closure of 
  $\F_2$ and let
  $$
  \chi:\mathrm{Grass}(2,\F^4)\longrightarrow\PP^4(\F)
  $$
  be the pole placement map associated to $M(s)$ over $\F$. 
  $\chi$ is a morphism, since $M(s)$ is
  non-degenerate.  We will now show that the restriction of
  $\chi$ to $\F_2$-rational points
  $$
  \mathrm{Grass}(2,\F_2^4)\longrightarrow\PP^4(\F_2)
  $$
  is never surjective.
  
  Let $\mathrm{rowsp}(A)\in \mathrm{Grass}(2,\F_2^4)$. Denote by
  $A_{i,j}$ the determinant of the sub-matrix of $A$ consisting of
  columns $i$ and $j$. Then:
  $$
  \chi(A)=\left[\sum_{i<j}\chi_{ijk}A_{i,j}\right]_{k=0,\ldots,4}$$
  where $$\det\left[\begin{array}{c} M(s)\\

      A\end{array}\right]=\sum_{k=0}^4\sum_{i<j}\chi_{ijk}A_{i,j}s^k.
  $$
  Since the system is non-degenerate, the $5\times 6$ matrix:
  $$
  C=\left[\begin{array}{ccc}
      \chi_{120} & \ldots & \chi_{340} \\
      \vdots & & \vdots \\
      \chi_{124} & \ldots & \chi_{344}
  \end{array}
\right]$$ has full rank, hence its kernel is 1-dimensional and
generated by a unique element of $\F_2^6$. By non-degeneracy, the
generator of the kernel corresponds to a point in $\PP^6(\F_2)$ 
which does not belong to $\mathrm{Grass}(2,\F^4)$. 
Hence we have the following
possibilities for the generator of $\ker C$:
$$(1,0,0,0,0,1),(0,1,0,0,1,0),(0,0,1,1,0,0),(1,1,1,1,1,1),$$
$$(1,1,0,0,0,1),(1,0,1,0,0,1),(1,0,0,1,0,1),(1,0,0,0,1,1),$$
$$(1,1,0,0,1,0),(0,1,1,0,1,0),(0,1,0,1,1,0),(0,1,0,0,1,1),$$
$$(1,0,1,1,0,0),(0,1,1,1,0,0),(0,0,1,1,1,0),(0,0,1,1,0,1),$$
$$(0,0,1,1,1,1),(0,1,0,1,1,1),(0,1,1,0,1,1),(0,1,1,1,0,1),$$
$$(1,0,0,1,1,1),(1,0,1,0,1,1),(1,0,1,1,1,0),(1,1,0,1,0,1),$$
$$(1,1,0,1,1,0),(1,1,1,0,0,1),(1,1,1,0,1,0),(1,1,1,1,0,0).$$
Observe that the problem is symmetric with respect to the
following changes of basis of $\F_2^6=\langle
e_{12},\ldots,e_{34}\rangle$ (which correspond to automorphisms
of $\mathrm{Grass}(2,\F^4)$) and composition thereof:\begin{itemize}
\item exchange $e_{12}$ and $e_{34}$ and leave the rest
  unaltered,
\item exchange $e_{13}$ and $e_{24}$ and leave the rest
  unaltered,
\item exchange $e_{14}$ and $e_{23}$ and leave the rest
  unaltered,
\item exchange $e_{12}$ and $e_{13}$, exchange $e_{34}$ and
  $e_{24}$,
\item exchange $e_{12}$ and $e_{14}$, exchange $e_{34}$ and
  $e_{23}$,
\item exchange $e_{13}$ and $e_{14}$, exchange $e_{24}$ and
  $e_{23}$.
\end{itemize}
Hence, reducing to the analysis of the following possibilities is
non-restrictive:
$$(1,0,0,0,0,1),(1,1,0,0,0,1),(0,0,1,1,1,1),(1,1,1,1,1,1).$$
Up
to a change of coordinates in $\PP^4$, we may assume that the
corresponding matrix $C$ is respectively:
$$\left[\begin{array}{cccccc}
    1 & 0 & 0 & 0 & 0 & 1 \\
    0 & 1 & 0 & 0 & 0 & 0 \\
    0 & 0 & 1 & 0 & 0 & 0 \\
    0 & 0 & 0 & 1 & 0 & 0 \\
    0 & 0 & 0 & 0 & 1 & 0 \\
  \end{array}\right],
\left[\begin{array}{cccccc}
    1 & 0 & 0 & 0 & 0 & 1 \\
    0 & 1 & 0 & 0 & 0 & 1 \\
    0 & 0 & 1 & 0 & 0 & 0 \\
    0 & 0 & 0 & 1 & 0 & 0 \\
    0 & 0 & 0 & 0 & 1 & 0 \\
  \end{array}\right],
\left[\begin{array}{cccccc}
    1 & 0 & 0 & 0 & 0 & 0 \\
    0 & 1 & 0 & 0 & 0 & 0 \\
    0 & 0 & 1 & 0 & 0 & 1 \\
    0 & 0 & 0 & 1 & 0 & 1 \\
    0 & 0 & 0 & 0 & 1 & 1 \\
  \end{array}\right],
\left[\begin{array}{cccccc}
    1 & 0 & 0 & 0 & 0 & 1 \\
    0 & 1 & 0 & 0 & 0 & 1 \\
    0 & 0 & 1 & 0 & 0 & 1 \\
    0 & 0 & 0 & 1 & 0 & 1 \\
    0 & 0 & 0 & 0 & 1 & 1 \\
  \end{array}\right].$$
Analyzing each case, it is now easy to prove that the
corresponding $\chi$ is not onto. E.g., in the first case we have:
$$
\chi(A)=\left[A_{1,2}+A_{3,4},A_{1,3},A_{1,4},A_{2,3},A_{2,4}\right]
$$
which is surjective if and only if the equations
$$
A_{1,2}+A_{3,4}=\alpha_0, A_{1,3}=\alpha_1,
A_{1,4}=\alpha_2,A_{2,3}=\alpha_3, A_{2,4}=\alpha_4,
$$
$$
A_{1,2}A_{3,4}+A_{1,3}A_{2,4}+A_{1,4}A_{2,3}=0
$$
have a solution in $\F_2^6$ for any choice of
$[\alpha_0:\ldots:\alpha_4]\in\PP^4(\F_2)$.  Letting $x=A_{1,2}$,
the equations reduce to:
$$
x^2+\alpha_0 x+\alpha_1 \alpha_4+\alpha_2 \alpha_3=0,
$$
which has no solution over $\F_2$ for
$\alpha_0=\alpha_1=\alpha_2=\alpha_4=1, \alpha_3=0$.

In the second case we have:
$$
\chi(A)=\left[A_{1,2}+A_{3,4},A_{1,3}+A_{3,4},A_{1,4},A_{2,3},A_{2,4}\right]
$$
which is surjective if and only if the equations
$$
A_{1,2}+A_{3,4}=\alpha_0, A_{1,3}+A_{3,4}=\alpha_1,
A_{1,4}=\alpha_2,A_{2,3}=\alpha_3, A_{2,4}=\alpha_4,
$$
$$A_{1,2}A_{3,4}+A_{1,3}A_{2,4}+A_{1,4}A_{2,3}=0$$
have a
solution in $\F_2^6$ for any choice of
$[\alpha_0:\ldots:\alpha_4]\in\PP^4(\F_2)$.  Letting $x=A_{3,4}$,
the equations reduce to:
$$
x^2+(\alpha_0+\alpha_4)x+\alpha_1 \alpha_4+\alpha_2
\alpha_3=0,
$$
which has no solution over $\F_2$ for
$\alpha_0=\alpha_1=\alpha_2=\alpha_3=1, \alpha_4=0$.

In the third case we have:
$$
\chi(A)=\left[A_{1,2},A_{1,3},A_{1,4}+A_{3,4},
  A_{2,3}+A_{3,4},A_{2,4}+A_{3,4}\right]
$$
which is surjective if and only if the equations
$$
A_{1,2}=\alpha_0, A_{1,3}=\alpha_1,
A_{1,4}+A_{3,4}=\alpha_2,A_{2,3}+A_{3,4}=\alpha_3,
A_{2,4}+A_{3,4}=\alpha_4,
$$
$$
A_{1,2}A_{3,4}+A_{1,3}A_{2,4}+A_{1,4}A_{2,3}=0
$$
have a solution in $\F_2^6$ for any choice of
$[\alpha_0:\ldots:\alpha_4]\in\PP^4(\F_2)$.  Letting $x=A_{3,4}$,
the equations reduce to:
$$x^2+(\alpha_0+\alpha_1+\alpha_2+\alpha_3)x+\alpha_1
\alpha_4+\alpha_2 \alpha_3=0,$$
which has no solution over $\F_2$
for $\alpha_0=\alpha_2=\alpha_3=0, \alpha_1=\alpha_4=1$.

In the last case we have:
$$
\chi(A)=\left[A_{1,2}+A_{3,4},A_{1,3}+A_{3,4},A_{1,4}+A_{3,4},
  A_{2,3}+A_{3,4},A_{2,4}+A_{3,4}\right]
$$
which is surjective if and only if the equations
$$
A_{1,2}+A_{3,4}=\alpha_0, A_{1,3}+A_{3,4}=\alpha_1,
A_{1,4}+A_{3,4}=\alpha_2,A_{2,3}+A_{3,4}=\alpha_3,
A_{2,4}+A_{3,4}=\alpha_4,
$$
$$
A_{1,2}A_{3,4}+A_{1,3}A_{2,4}+A_{1,4}A_{2,3}=0
$$
have a solution in $\F_2^6$ for any choice of
$[\alpha_0:\ldots:\alpha_4]\in\PP^4(\F_2)$.  Letting $x=A_{3,4}$,
the equations reduce to:
$$
x^2+(\alpha_0+\alpha_1+\alpha_2+\alpha_3+\alpha_4)x+\alpha_1
\alpha_4+\alpha_2 \alpha_3=0,
$$
which has no solution over $\F_2$ for
$\alpha_0=\alpha_1=\alpha_4=1, \alpha_2=\alpha_3=0$.
\qed
\end{proof}

\begin{acknowledgement}
  The work of both authors is supported by the Swiss National
  Science Foundation through grants \#123393, \#113251, and 
  \#107887. 
\end{acknowledgement}


\def\cprime{$'$} \def\cprime{$'$} \def\cprime{$'$}


\begin{thebibliography}{10}
\providecommand{\url}[1]{{#1}}
\providecommand{\urlprefix}{URL }
\expandafter\ifx\csname urlstyle\endcsname\relax
  \providecommand{\doi}[1]{DOI~\discretionary{}{}{}#1}\else
  \providecommand{\doi}{DOI~\discretionary{}{}{}\begingroup
  \urlstyle{rm}\Url}\fi

\bibitem{br81}
Brockett, R.W., Byrnes, C.I.: Multivariable {N}yquist criteria, root loci and
  pole placement: A geometric viewpoint.
\newblock IEEE Trans. Automat. Control \textbf{AC-26}, 271--284 (1981)

\bibitem{by89}
Byrnes, C.I.: Pole assignment by output feedback.
\newblock In: H.~Nijmeijer, J.M. Schumacher (eds.) Three Decades of
  Mathematical System Theory, Lecture Notes in Control and Information Sciences
  \# 135, pp. 31--78. Springer Verlag (1989)

\bibitem{ge85}
Gessel, I., Viennot, G.: Binomial determinants, paths, and hook length
  formulae.
\newblock Adv. in Math. \textbf{58}(3), 300--321 (1985)

\bibitem{kl76}
Kleiman, S.L.: Problem 15: Rigorous foundations of {S}chubert's enumerative
  calculus.
\newblock In: Proceedings of Symposia in Pure Mathematics, vol.~28, pp.
  445--482. Am. Math. Soc. (1976)

\bibitem{ma78}
Martin, C.F., Hermann, R.: Applications of algebraic geometry to system theory:
  The {McM}illan degree and {K}ronecker indices as topological and holomorphic
  invariants.
\newblock SIAM J. Control Optim. \textbf{16}, 743--755 (1978)

\bibitem{mu76}
Mumford, D.: Algebraic Geometry {I}: Complex Projective Varieties.
\newblock Springer Verlag, Berlin, New York (1976)

\bibitem{ro90}
Rosenthal, J.: Geometric methods for feedback stabilization of multivariable
  linear systems.
\newblock Ph.D. thesis, Arizona State University (1990)

\bibitem{ro94}
Rosenthal, J.: On dynamic feedback compensation and compactification of
  systems.
\newblock SIAM J. Control Optim. \textbf{32}(1), 279--296 (1994)

\bibitem{ro97a1}
Rosenthal, J., Schumacher, J.M.: Realization by inspection.
\newblock IEEE Trans. Automat. Contr. \textbf{AC-42}(9), 1257--1263 (1997)

\bibitem{ro98}
Rosenthal, J., Sottile, F.: Some remarks on real and complex output feedback.
\newblock Systems \& Control Letters \textbf{33}(2), 73--80 (1998)

\bibitem{sc1886}
Schubert, H.: {A}nzahlbestimmung f{\"u}r lineare {R}{\"a}ume beliebiger
  {D}imension.
\newblock Acta Math. \textbf{8}, 97--118 (1886)

\bibitem{sc1891}
Schubert, H.: {B}eziehungen zwischen den linearen {R}{\"a}umen auferlegbaren
  charakteristischen {B}edingungen.
\newblock Math. Ann. \textbf{38}, 598--602 (1891)

\bibitem{sh94b1}
Shafarevich, I.R.: Basic algebraic geometry. 1, second edn.
\newblock Springer-Verlag, Berlin (1994).
\newblock Varieties in projective space, Translated from the 1988 Russian
  edition and with notes by Miles Reid

\bibitem{wa92}
Wang, X.: Pole placement by static output feedback.
\newblock Journal of Math. Systems, Estimation, and Control \textbf{2}(2),
  205--218 (1992)

\bibitem{wa96}
Wang, X.: Grassmannian, central projection and output feedback pole assignment
  of linear systems.
\newblock IEEE Trans. Automat. Contr. \textbf{AC-41}(6), 786--794 (1996)

\end{thebibliography}
\end{document}